\documentclass{article}
\usepackage{amsfonts,latexsym} 
\usepackage{hyperref}
\newcounter{satznum}
\newtheorem{theorem}{Theorem}[satznum]

\newtheorem{lemma}[theorem]{Lemma}
\newtheorem{corollary}[theorem]{Corollary}

\newenvironment{acknowledgement}
 {\begin{trivlist}\item[]{\bf Acknowledgement.}}
 {\end{trivlist}}
\newenvironment{remark}
 {\begin{trivlist}\item[]{\bf Remark.}}
 {\end{trivlist}}
\newenvironment{remarks}
 {\begin{trivlist}\item[]{\bf Remarks.}}
 {\end{trivlist}}
\newenvironment{example}
 {\begin{trivlist}\item[]{\bf Example.}}
 {\end{trivlist}}

\newenvironment{proof}
 {\begin{trivlist}\item[]{\bf Proof.}}
 {\end{trivlist}}
{\makeatletter
\gdef\cz{{\mathbb C}} 
\gdef\nz{{\mathbb N}} 
\gdef\rz{{\mathbb R}} 
}
\begin{document}
   \section*{ON THE NUMBER OF ALLELIC TYPES FOR SAMPLES TAKEN FROM
   EXCHANGEABLE COALESCENTS WITH MUTATION}
   {\sc F. Freund and M. M\"ohle}\footnote{E-mail addresses: freund@math.uni-duesseldorf.de, moehle@math.uni-duesseldorf.de}
   Mathematisches Institut, Heinrich-Heine-Uni\-versit\"at D\"usseldorf,
   Universit\"atsstr. 1, 40225 D\"usseldorf, Germany
\begin{abstract}
   Let $K_n$ denote the number of types of a sample of size $n$
   taken from an exchangeable coalescent process ($\Xi$-coalescent)
   with mutation. A distributional recursion for the sequence
   $(K_n)_{n\in\nz}$ is derived. If the coalescent does not have
   proper frequencies, i.e., if the characterizing measure
   $\Xi$ on the infinite simplex $\Delta$ does not have mass at zero
   and satisfies $\int_\Delta |x|\Xi(dx)/(x,x)<\infty$,
   where $|x|:=\sum_{i=1}^\infty x_i$ and $(x,x):=\sum_{i=1}^\infty x_i^2$
   for $x=(x_1,x_2,\ldots)\in\Delta$,
   then $K_n/n$ converges weakly as $n\to\infty$ to a limiting
   variable $K$ which is characterized by an exponential integral of the
   subordinator associated with the coalescent process. For so-called
   simple measures $\Xi$ satisfying $\int_\Delta\Xi(dx)/(x,x)<\infty$
   we characterize the distribution of $K$ via a fixed-point equation.

   \vspace{2mm}

   \noindent Running head: Number of types for coalescents

   \vspace{2mm}

   \noindent Keywords: Coalescent; Distributional recursion; Fixed-point;
   Number of types; Simultaneous multiple collisions; Subordinator

   \vspace{2mm}

   \noindent AMS 2000 Mathematics Subject Classification:
            Primary 60C05;   
                    05C05    
            Secondary 60F05; 
                      92D15  
\end{abstract}
\subsection{Introduction and main results} \label{intro}
\setcounter{theorem}{0}
Exchangeable coalescents are Markovian processes with state space
${\cal E}$, the set of equivalence relations (partitions) on
$\nz:=\{1,2,\ldots\}$ with a block merging mechanism. The class
of exchangeable coalescents with multiple collisions has been
independently introduced by Pitman \cite{pitman} and Sagitov
\cite{sagitov}. These processes can be characterized by a finite
measure $\Lambda$ on the unit interval $[0,1]$ and are hence also
called $\Lambda$-coalescents. The best known example is the
Kingman coalescent where $\Lambda=\delta_0$ is the Dirac measure
in $0$. This coalescent allows only for binary mergers of ancestral
lineages. Another well studied coalescent is the Bolthausen-Sznitman
coalescent \cite{bolthausensznitman}, where $\Lambda$ is uniformly
distributed on $[0,1]$. The full class of exchangeable
coalescents allowing for simultaneous multiple collisions of
ancestral lineages was discovered by M\"ohle and Sagitov
\cite{moehlesagitov} and Schweinsberg \cite{schweinsberg2}.
Schweinsberg \cite{schweinsberg2} characterizes exchangeable
coalescents via a finite measure $\Xi$ on the infinite simplex
$\Delta:=\{x=(x_1,x_2,\ldots):x_1\geq x_2\geq\cdots\geq 0,
\sum_{i=1}^\infty x_i\leq 1\}$. For the following it is convenient
to decompose $\Xi=a\delta_0+\Xi_0$ with $a:=\Xi(\{0\})\in [0,\infty)$
and $\Xi_0$ having no atom at zero. Suppose that the coalescent is in a
state with $n$ blocks. Then each $(k_1,\ldots,k_j)$-collision
($k_1,\ldots,k_j\in\nz$ with $k_1+\cdots+k_j=n$,
$k_1\geq\cdots\geq k_j$ and $k_1\geq 2$)
is occurring at the rate (see \cite[Eq. (11)]{schweinsberg2})
\begin{eqnarray}
   &   & \hspace{-1cm}\phi_j(k_1,\ldots,k_j)
   \ = \ a\,1_{\{r=1,k_1=2\}}\nonumber\\
   &   & + \int_\Delta\sum_{l=0}^s {s\choose l}(1-|x|)^{s-l}
         \sum_{{i_1,\ldots,i_{r+l}\in\nz}\atop{\rm all\ distinct}}
         x_{i_1}^{k_1}\cdots x_{i_{r+l}}^{k_{r+l}}
         \frac{\Xi_0(dx)}{(x,x)}, \label{xirates1}
\end{eqnarray}
where $s:=|\{1\leq i\leq j:k_i=1\}|$, $r:=j-s$,
$|x|:=\sum_{i=1}^\infty x_i$ and $(x,x):=\sum_{i=1}^\infty x_i^2$
for $x=(x_1,x_2,\ldots)\in\Delta$. Note that $\phi_1(2)=\Xi(\Delta)$.

For $n\in\nz$ let $\varrho_n:{\cal E}\to{\cal E}_n$ denote the
natural restriction to the set ${\cal E}_n$ of all equivalence relations
on $\{1,\ldots,n\}$. Let $R=(R_t)_{t\ge 0}$ be a coalescent process
with simultaneous multiple collisions. The restricted coalescent
process $(\varrho_nR_t)_{t\ge 0}$ is usually interpreted as a
genealogical tree of a sample of $n$ individuals.

In the biological context it is natural to introduce mutations
into this model as follows. Assume that each individual has a
certain type. Independently of the genealogical tree mutations
occur along each branch of the tree according to a homogeneous
Poisson process with rate $r>0$. The infinitely many alleles model
is assumed, i.e., each mutation leads to a new type never seen
before in the sample.

Recently there is much interest in the study of functionals of
restricted coalescent processes $(\varrho_nR_t)_{t\ge 0}$, for
example the number of collisions
\cite{drmotaiksanovmoehleroesler2, gnedinyakubovich,
iksanovmarynychmoehle, iksanovmoehle1}, the time back to the most
recent common ancestor and the lengths of external branches
\cite{caliebeneiningerkrawczakroesler, delmasdhersinsirijegousse,
freundmoehle1}, the total branch length
\cite{drmotaiksanovmoehleroesler1} or the number of segregating
sites \cite{moehlesites}.

Further typical quantities of interest are $K_i(n)$, the number of
types which appear exactly $i$ times in a sample of size $n$, and
the summary statistics $K_n:=\sum_{i=1}^n K_i(n)$, the total
number of types in the sample. The most celebrated result in this
context is the Ewens sampling formula \cite{ewens} for the
distribution of the allele frequency spectrum
$(K_1(n),\ldots,K_n(n))$ under the Kingman coalescent. Recently
asymptotic results for the allele frequency spectrum have been
obtained by Berestycki, Berestycki and Schweinsberg
\cite{berestyckischweinsberg1,berestyckischweinsberg2} for
beta$(2-\alpha,\alpha)$-coalescents with parameter $1<\alpha<2$
and by Basdevant and Goldschmidt \cite{basdevantgoldschmidt} for
the Bolthausen-Sznitman coalescent \cite{bolthausensznitman}. Here
we are interested in the total number $K_n$ of types of a sample
of size $n\in\nz$ taken from a $\Xi$-coalescent with mutation rate
$r>0$. The motivation for our interest in $K_n$ is manifold. It is
an observable quantity and hence important for biological
and statistical applications. In combination with the results
of \cite{moehleewens} on the allele frequency spectrum and of
\cite{moehlesites} on the number of segregating sites, our
study of $K_n$ gives additional insight in the structure of
exchangeable coalescent trees. Our first result (Theorem \ref{main1}
below) provides a distributional recursion for the sequence
$(K_n)_{n\in\nz}$. In order to state the result we need to
introduce the rates
\begin{equation} \label{rates}
   g_{nk}\ :=\ \lim_{t\searrow 0}\frac{P(|\varrho_nR_t|=k)}{t},
   \quad n,k\in\nz, k<n,
\end{equation}
and the total rates
\begin{equation} \label{totalrates}
   g_n\ :=\ \lim_{t\searrow 0}\frac{P(|\varrho_nR_t|<n)}{t}
   \ =\ \sum_{k=1}^{n-1}g_{nk},\qquad n\in\nz.
\end{equation}
The total rates $g_n$, $n\in\nz$, can be expressed in terms of
the measure $\Xi=a\delta_0+\Xi_0$ as (see Schweinsberg
\cite[p.~36, Eq.~(70)]{schweinsberg2})
\begin{equation} \label{totalrates2}
   g_n\ =\ a{n\choose 2} + \int_\Delta
   \Bigg(
      1-(1-|x|)^n-\sum_{j=1}^n{n\choose j}
      \sum_{{i_1,\ldots,i_j\in\nz}\atop{\rm all\,distinct}}
      x_{i_1}\cdots x_{i_j}
   \Bigg)\frac{\Xi_0(dx)}{(x,x)}.
\end{equation}
A similar argument shows that the rates (\ref{rates}) are given as
\begin{equation} \label{rates2}
   g_{nk}\ =\ a{n\choose 2}1_{\{k=n-1\}} +
   \int_\Delta \sum_{j=1}^k f_{nkj}(x)\frac{\Xi_0(dx)}{(x,x)},
   \quad n,k\in\nz, k<n,
\end{equation}
with
$$
f_{nkj}(x)\ :=\ \sum_{{i_1,\ldots,i_j\in\nz}\atop{\rm all\,distinct}}
\sum_{{n_1,\ldots,n_j\in\nz}\atop{n_1+\cdots+n_j=n-k+j}}
\frac{n!}{(k-j)!n_1!\cdots n_j!}(1-|x|)^{k-j} x_{i_1}^{n_1}\cdots x_{i_j}^{n_j}
$$
for $n,k\in\nz$ with $k<n$ and $j\in\{1,\ldots,k\}$.
The $\Lambda$-coalescent occurs, if the measure $\Xi$ is concentrated on
the points $x=(u,0,0,\ldots)\in\Delta$ with $u\in [0,1]$ and can be hence
considered as a measure $\Lambda$ on the unit interval $[0,1]$. In this case
only the index $j=1$ contributes to the sum below the integral in (\ref{rates2})
and from $f_{nk1}(u,0,0,\ldots)={n\choose{k-1}}(1-u)^{k-1}u^{n-k+1}$
it follows that (\ref{rates2}) takes the form 
\begin{equation}
   g_{nk}\ =\ {n\choose{k-1}}\int_{[0,1]}u^{n-k-1}(1-u)^{k-1}\,\Lambda(du),
   \quad n,k\in\nz, k<n.
\end{equation}
Similarly, for the $\Lambda$-coalescent the total rates (\ref{totalrates2}) are
given as
\begin{equation}
   g_n\ =\
   \int_{[0,1]}\frac{1-(1-u)^n-nu(1-u)^{n-1}}{u^2}\,\Lambda(du),\qquad n\in\nz.
\end{equation}
Our first main result is the following distributional recursion
for the number of types $K_n$.
\begin{theorem} \label{main1}
   The sequence $(K_n)_{n\in\nz}$ satisfies the distributional recursion
   \begin{equation} \label{knrec}
   K_1=1\mbox{ and } K_n\ \stackrel{d}{=}\ B_n(K_{n-1}+1) + (1-B_n)K_{I_n},
   \quad n\in\{2,3,\ldots\},
   \end{equation}
   where $B_n$ is a Bernoulli variable independent of $(K_2,\ldots,K_{n-1},I_n)$
   with distribution
   $$
   P(B_n=1)\ =\ 1-P(B_n=0)\ =\ \frac{nr}{g_n+nr},\quad n\in\nz,
   $$
   and $I_n$ is a random variable independent of $(K_2,\ldots,K_{n-1})$
   with distribution
   \begin{equation} \label{indist}
      r_{nk}\ :=\ P(I_n=k)\ =\ \frac{g_{nk}}{g_n},\qquad n,k\in\nz, k<n.
   \end{equation}
\end{theorem}
Note that $I_n$ is the number of equivalence classes (blocks) of
the restricted coalescent process $(\varrho_nR_t)_{t\ge 0}$ after
its first jump.

The proof of Theorem \ref{main1} given in Section \ref{recursion}
involves a combination of what Kingman \cite{kingman1} calls
{\em natural coupling} and {\em temporal coupling}. The main argument
of the proof is the same as that used in \cite{moehleewens} and
\cite{moehlesites} for deriving similar recursions for the allele
frequency spectrum and the number of segregating sites.
The recursion for the summary statistics $K_n$ is simpler than that
for the allele frequency spectrum presented in \cite{moehleewens}.
It is therefore more useful to compute the distribution
and other related functionals of the distribution of $K_n$ for
moderate values of $n$ in reasonable time. Moreover, Theorem
\ref{main1} is valid for any arbitrary $\Xi$-coalescent.

Our second result (Theorem \ref{main2} below) concerns measures $\Xi$
satisfying
\begin{equation} \label{cond}
   \Xi(\{0\})\ =\ 0
   \quad\mbox{and}\quad
   \int_{\Delta\setminus\{0\}}\frac{|x|}{(x,x)}\Xi(dx)\ <\ \infty.
\end{equation}
Recall that $|x|:=\sum_{i=1}^\infty x_i$ and that
$(x,x):=\sum_{i=1}^\infty x_i^2$ for
$x=(x_1,x_2,\ldots)\in\Delta$. Note that (\ref{cond}) prevents
$\Xi$ from having too much mass near zero. Schweinsberg
\cite[Prop. 30]{schweinsberg2} showed that the $\Xi$-coalescent
does not have proper frequencies if and only if (\ref{cond})
holds. Not having proper frequencies is equivalent to having a
positive fraction of singleton blocks with positive probability,
which is actually most important for our convergence result
presented in Theorem \ref{main2} below. For the special class of
coalescent processes with multiple collisions ($\Lambda$-coalescents),
Eq. (\ref{cond}) takes the form
\begin{equation} \label{cond2}
   \Lambda(\{0\})\ =\ 0
   \quad\mbox{and}\quad
   \int_{(0,1]} u^{-1}\Lambda(du)\ <\ \infty.
\end{equation}
Pitman \cite[Theorem 8]{pitman} already showed that the
$\Lambda$-coalescent does not have proper frequencies if and only
if (\ref{cond2}) holds. Condition (\ref{cond2}) excludes important
examples such as the Kingman coalescent and the
Bolthausen-Sznitman coalescent \cite{bolthausensznitman}. However,
it includes for example all beta$(a,b)$-coalescents with
parameters $a>1$ and $b>0$, which are studied in more detail in
Section \ref{examples}. Note that Theorem \ref{main2} covers a
substantial class of $\Xi$-coalescents.
\begin{theorem} \label{main2}
   Suppose that the characterizing measure $\Xi$ of the exchangeable
   coalescent $(R_t)_{t\ge 0}$ satisfies (\ref{cond}). Then $K_n/n$
   converges weakly as $n\to\infty$ to $K:=r\int_0^\infty e^{-rt}e^{-X_t}dt$,
   where $X=(X_t)_{t\ge 0}$ is a subordinator with Laplace exponent
   $$
   \Phi(\eta)\ =\ \int_{\Delta\setminus\{0\}}(1-(1-|x|)^\eta)\frac{\Xi(dx)}{(x,x)},
   \qquad\eta\ge 0.
   $$
   The limiting variable $K$ has moments
   \begin{equation} \label{kj}
      {\rm E}(K^j)\ =\ \frac{r^jj!}{(r+\Phi(1))(2r+\Phi(2))\cdots(jr+\Phi(j))},
      \qquad j\in\nz.
   \end{equation}
\end{theorem}
We will see that the subordinator $X$ appearing in Theorem
\ref{main2} is related to the frequency $S_t$ of singletons of $R_t$ via
$X_t=-\log S_t$, $t\ge 0$.
Our proof of Theorem \ref{main2} is not based on the recursion
presented in Theorem \ref{main1}. It is rather a consequence of the
chain of inequalities
\begin{equation}
   M_n\ \le\ K_n\ \leq\ N_n+1,
\end{equation}
where $M_n$ denotes the number of mutated external branches and
$N_n$ denotes the total number of mutated branches of the
restricted coalescent tree $(\varrho_nR_t)_{t\ge 0}$ respectively.
Here we call a branch mutated, if it is affected by at least one
mutation. In a first step it is shown in Section \ref{mn} that
Theorem \ref{main2} is valid with $K_n$ replaced by the lower
bound $M_n$. Afterwards in Section \ref{nn} it is verified that
$(N_n-M_n)/n\to 0$ in probability (even in $L^1$), which completes
the proof of Theorem \ref{main2} and in addition shows that
Theorem \ref{main2} remains valid with $K_n$ replaced by $N_n$.
Note that, if $K_{n,1}$ denotes the number of types which appear
exactly once in the sample of size $n$, then $M_n\le K_{n,1}\le
K_n$, and, consequently, Theorem \ref{main2} remains also valid
with $K_n$ replaced by $K_{n,1}$.

Theorem \ref{main2} leaves open the question about the asymptotical
behavior of $K_n$ for the important class of $\Xi$-coalescents which
do not satisfy condition (\ref{cond}). As mentioned before, some results for particular
$\Lambda$-coalescents are known (\cite{basdevantgoldschmidt},
\cite{berestyckischweinsberg1}, \cite{berestyckischweinsberg2}, \cite{ewens},
\cite{moehleewens}), however,
the problem concerning the asymptotical behavior of $K_n$ for the full class
of $\Xi$-coalescents remains open.

%
\subsection{A recursion for the number of types}\label{recursion}
\setcounter{theorem}{0}

The proof of Theorem \ref{main1} is based on two fundamental properties
of coalescent processes which Kingman \cite{kingman1} calls {\em natural
coupling} and {\em temporal coupling}.

{\em Natural coupling} states
the following. Suppose a genealogy of a sample of size $n\in\nz$ governed
by a $\Xi$-coalescent is given. If a sub-sample of size $m\in\{1,\ldots,n-1\}$
of this sample is taken, i.e., if $n-m$ individuals are removed
from the sample, then the genealogical tree of the remaining sample
of size $m$ is governed by the same $\Xi$-coalescent. This consistency
relation between different sample sizes is one of the fundamental properties
of exchangeable coalescents. It is in fact needed in order to prove the
existence of exchangeable coalescent processes with state space ${\cal E}$
via Kolmogoroff's extension theorem.

The second property, called {\em temporal coupling} states the following.
Consider a restricted coalescent process $(R_t^{(n)})_{t\ge 0}:=(\varrho_nR_t)_{t\ge 0}$
and let $T_n:=\inf\{t>0:R_t^{(n)}\neq R_0^{(n)}\}$ denote the time of its first
jump. If you identify individuals which belong after that first jump
to the same equivalence class, then the process started at time $T_n$ is
distributed as a coalescent with sample size $|R_{T_n}^{(n)}|$. Mathematically
this property essentially boils down to the strong Markov property.

We will now verify Theorem \ref{main1}.

\vspace{2mm}

{\bf Proof of Theorem \ref{main1}}.
   The recursion (\ref{knrec}) is equivalent to $P(K_1=1)=1$ and
   \begin{equation} \label{distrec}
      P(K_n=k)\ =\ \frac{nr}{g_n+nr}P(K_{n-1}=k-1) + \frac{g_n}{g_n+nr}\sum_{i=k}^{n-1} r_{ni}P(K_i=k)
   \end{equation}
   for $n\in\{2,3,\ldots\}$ and $k\in\{1,\ldots,n\}$. We verify
   (\ref{distrec}) in analogy to the proofs presented in \cite{moehleewens}
   by looking at the first event (either a coalescence or a mutation)
   which happens backwards in time.

   The time $W_n$ back to the first mutation is exponentially distributed
   with parameter $nr$. The time $T_n$ back to the first coalescence is
   independent of $W_n$ and exponentially distributed with parameter $g_n$.
   Thus, the first event backwards in time is a mutation with probability
   $P(W_n<T_n)=nr/(g_n+nr)$, and a coalescence with the complementary probability
   $P(T_n<W_n)=g_n/(g_n+nr)$. Note that these two probabilities appear on the
   right hand side in (\ref{distrec}).

   Assume that the first event backwards in time is a
   mutation. If we disregard the individual which is affected by this
   mutation, the number of types decreases by one. Moreover, from
   the natural coupling property it follows that the remaining tree
   is distributed as a coalescent restricted to the set $\{1,\ldots,n-1\}$.
   This argument explains the appearance of the probability $P(K_{n-1}=k-1)$
   on the right hand side in (\ref{distrec}).

   If the first event backwards in time is a coalescence, then at the time
   of that coalescence event, the coalescent process jumps to a partition
   with $i$ blocks, $i\in\{1,\ldots,n-1\}$, with
   probability $r_{ni}=g_{ni}/g_n$. By the temporal coupling property, the
   coalescent process stopped at that time is distributed as a coalescent
   restricted to the set $\{1,\ldots,i\}$. As the number of types is
   not affected by a coalescence, the appearance of the sum on the right
   hand side in (\ref{distrec}) is explained. Note that it suffices to
   run the sum from $k$ to $n-1$ as $P(K_i=k)=0$ for $i<k$.\hfill$\Box$

\begin{remarks}
   1. In terms of the generating function $f_n(s):={\rm E}(s^{K_n})$, $n\in\nz$,
   $s\in\cz$, the recursion (\ref{knrec}) (or (\ref{distrec})) is equivalent
   to $f_1(s)=s$ and
   \begin{equation} \label{pgfrec}
      (g_n+nr)f_n(s)\ =\ nrsf_{n-1}(s) + \sum_{k=1}^{n-1} g_{nk} f_k(s),
      \quad n\in\{2,3,\ldots\}, s\in\cz,
   \end{equation}
   a formula which follows (at least for coalescent processes with
   multiple collisions) also by taking $s_1=\cdots=s_n=:s$ in
   Eq.~(4) of \cite{moehleoberwolf}.

   2. The recursion (\ref{distrec}) for the distribution of $K_n$ is useful
   to compute the probabilities $P(K_n=k)$ successively for $k=n,n-1,\ldots,1$.
   For example, for $k=n$ it follows that
   $(g_n+nr)P(K_n=n)=nrP(K_{n-1}=n-1)$ and, therefore,
   $$
   P(K_n=n)\ =\ \prod_{i=2}^n \frac{ir}{g_i+ir}
   \ =\ \frac{r^{n-1}n!}{\prod_{i=2}^n(g_i+ir)},\quad n\in\nz.
   $$
   Note that $P(K_n=n)$ is the probability to have only singletons
   in the sample of size $n$.
\end{remarks}
\begin{example} (Kingman coalescent)
   For the Kingman coalescent ($\Lambda=\delta_0$) we have $I_n\equiv n-1$,
   $g_n=g_{n,n-1}=n(n-1)/2$ and $g_{ni}=0$ for $i\in\{1,\ldots,n-2\}$.
   The recursion (\ref{knrec}) reduces to
   $K_n\stackrel{d}{=}B_n+K_{n-1}$. Therefore,
   $K_n\stackrel{d}{=}\sum_{i=1}^n B_i$, $n\in\nz$, where $B_1,B_2,\ldots$
   are independent Bernoulli variables with
   $P(B_n=1)=nr/(g_n+nr)=\theta/(\theta+n-1)$, $n\in\nz$,
   with $\theta:=2r$. It follows easily that
   $P(K_n=k)=\theta^ks(n,k)/[\theta]_n$, where
   $[\theta]_n:=\theta(\theta+1)\cdots(\theta+n-1)$ and
   the $s(n,k)$ denote
   the absolute Stirling numbers of the first kind.
   Moreover, ${\rm E}(K_n)=\theta\sum_{i=0}^{n-1}1/(\theta+i)\sim
   \theta\log n$ and ${\rm Var}(K_n)=\theta\sum_{i=1}^{n-1}i/(\theta+i)^2
   \sim\theta\log n$. By the Lindeberg-Feller central limit theorem,
   $(K_n-\theta\log n)/\sqrt{\theta\log n}$
   is asymptotically standard normal distributed. All these
   results are of course well known and go at least back to
   the seminal work of Ewens \cite{ewens}.
\end{example}
\begin{example} (Star-shaped coalescent)
   For the star-shaped coalescent ($\Lambda=\delta_1$) we have
   $I_n\equiv 1$,
   $g_{n1}=g_n=1$ and $g_{ni}=0$ for $i\in\{2,\ldots,n-1\}$. Therefore,
   (\ref{pgfrec}) reduces to $(1+nr)f_n(s)=nrsf_{n-1}(s)+s$,
   $n\in\{2,3,\ldots\}$, $s\in\cz$.
   We refer to \cite[Section 4]{moehleewens} for more details. In particular,
   in \cite{moehleewens} it is shown that $K_n/n$ converges almost surely
   to a limiting random variable $K$, beta distributed with parameter
   $1$ and $1/r$, that is $P(K>x)=(1-x)^{1/r}$, $0<x<1$.
\end{example}
\begin{remark} (Recursion for the factorial moments of $K_n$)
Taking the $j$th derivative with respect to $s$ in (\ref{pgfrec}) and
applying the Leibniz rule yields
$$
(g_n+nr)f_n^{(j)}(s)\ =\ nr\big(sf_{n-1}^{(j)}(s)+jf_{n-1}^{(j-1)}(s)\big)
+ \sum_{k=1}^{n-1} g_{nk} f_k^{(j)}(s)
$$
for $n\in\{2,3,\ldots\}$, $j\in\nz$ and $s\in\cz$.
For $n\in\nz$ and $j\in\nz_0$
let $\mu_n^{(j)}:={\rm E}((K_n)_j)={\rm E}(K_n(K_n-1)\cdots(K_n-j+1))$ denote
the $j$th descending factorial moment of $K_n$. Taking the limit $s\to 1$
it follows that
$$
(g_n+nr)\mu_n^{(j)}
\ =\ nr\big(\mu_{n-1}^{(j)}+j\mu_{n-1}^{(j-1)}\big) + \sum_{k=1}^{n-1}g_{nk}\mu_k^{(j)},
\quad n\in\{2,3,\ldots\}, j\in\nz.
$$
This recursion with initial condition $\mu_1^{(j)}=\delta_{j1}$ (Kronecker symbol)
is useful to compute the factorial moments of $K_n$. For example, for
$j=n$ we have $(g_n+nr)\mu_n^{(n)}=n^2r\mu_{n-1}^{(n-1)}$ and, therefore,
$$
\mu_n^{(n)}\ =\ \prod_{i=2}^n \frac{i^2r}{g_i+ir}
\ =\ \frac{r^{n-1}(n!)^2}{\prod_{i=2}^n(g_i+ir)},\quad n\in\nz,
$$
a result which also follows from $\mu_n^{(n)}=n!P(K_n=n)$. In particular,
the first moment $\mu_n:=\mu_n^{(1)}={\rm E}(K_n)$ follows the recursion
$\mu_1=1$ and
\begin{equation} \label{meanrec}
   (g_n+nr)\mu_n\ =\ nr(\mu_{n-1}+1) + \sum_{k=1}^{n-1} g_{nk}\mu_k,
   \qquad n\in\{2,3,\ldots\}.
\end{equation}
\end{remark}
It seems to be non-trivial to solve any of these recursions except for
the Kingman coalescent ($\Lambda=\delta_0$) and the star-shaped coalescent
($\Lambda=\delta_1$). We therefore focus on asymptotic results for $K_n$
as the sample size $n$ tends to infinity.
\subsection{The number of mutated external branches} \label{mn}
\setcounter{theorem}{0}
We say that a branch of the restricted coalescent tree
$(\varrho_nR_t)_{t\ge 0}$ is mutated, if it is affected by at least
one mutation. In this section we study the asymptotics of the
number $M_n$ of mutated external branches of $(\varrho_nR_t)_{t\ge 0}$
under the assumption that the measure $\Xi$ satisfies the condition
(\ref{cond}).
\begin{lemma} \label{mnconv}
   Suppose that the characterizing measure $\Xi$ of the exchangeable
   coalescent process $R=(R_t)_{t\ge 0}$ satisfies (\ref{cond}). Then,
   $M_n/n\stackrel{d}{\to}M$ as $n\to\infty$, where $M$ is a random
   variable uniquely determined by its moments
   $$
   {\rm E}(M^k)\ =\ {\rm E}\Big(\prod_{i=1}^k (1-e^{-rL_i})\Big),
   \quad k\in\nz,
   $$
   with $L_i:=\sup\{t>0:\mbox{$\{i\}$ is a block of $R_t$}\}$, $i\in\nz$.
\end{lemma}
\begin{proof}
   For $n\in\nz$ and $i\in\{1,\ldots,n\}$ let
   $$
   L_{n,i}\ :=\
   \sup\{t>0:\mbox{$\{i\}$ is a block of $\varrho_nR_t$}\}
   $$
   denote the length of the $i$th external branch of the restricted coalescent
   tree $(\varrho_nR_t)_{t\ge 0}$. Fix $k\in\nz$ and
   $t_1,\ldots,t_k\in [0,\infty)$. For $n\ge k$ we have
   \begin{eqnarray*}
      &   & \hspace{-15mm}P(L_{n,1}>t_1,\ldots,L_{n,k}>t_k)\\
      & = & P(\mbox{$\{1\}$ is a block of $\varrho_nR_{t_1}$},\ldots,
            \mbox{$\{k\}$ is a block of $\varrho_nR_{t_k}$})\\
      & \to & P(\bigcap_{n\in\nz}\{
            \mbox{$\{1\}$ is a block of $\varrho_nR_{t_1}$},\ldots,
            \mbox{$\{k\}$ is a block of $\varrho_nR_{t_k}$}\})\\
      & = & P(\mbox{$\{1\}$ is a block of $R_{t_1}$},\ldots,
            \mbox{$\{k\}$ is a block of $R_{t_k}$})\\
      & = & P(L_1>t_1,\ldots,L_k>t_k).
   \end{eqnarray*}
   Thus, for all $k\in\nz$, $(L_{n,1},\ldots,L_{n,k})\stackrel{d}{\to}(L_1,\ldots,L_k)$
   as $n\to\infty$.

   For $n\in\nz$ and $i\in\{1,\ldots,n\}$ let $E_{n,i}$ denote the event
   that the $i$th external branch of the restricted tree $(\varrho_nR_t)_{t\ge 0}$
   is affected by at least one mutation. Conditional on the lengths
   $L_{n,1},\ldots,L_{n,n}$ of
   the external branches, the mutation Poisson process with parameter
   $r>0$ acts independently on all these branches.
   Thus, for fixed $j\in\nz$ we have
   \begin{eqnarray*}
      P(E_{n,1}\cap\cdots\cap E_{n,j})
      & = & E(P(E_{n,1}\cap\cdots\cap E_{n,j}|L_{n,1},\ldots,L_{n,j}))\\
      & = & E(P(E_{n,1}|L_{n,1})\cdots P(E_{n,j}|L_{n,j}))\\
      & = & E((1-e^{-rL_{n,1}})\cdots(1-e^{-rL_{n,j}}))\\
      & \to & E((1-e^{-rL_1})\cdots(1-e^{-rL_j})).
   \end{eqnarray*}
   From $M_n=\sum_{i=1}^n 1_{E_{n,i}}$ it follows that
   \begin{eqnarray*}
      {\rm E}(M_n^k)
      & = & {\rm E}\Big(\Big(\sum_{i=1}^n 1_{E_{n,i}}\Big)^k\Big)
      \ = \ \sum_{i_1,\ldots,i_k=1}^n {\rm E}(1_{E_{n,i_1}}\cdots 1_{E_{n,i_k}})\\
      & = & \sum_{i_1,\ldots,i_k\in\{1,\ldots,n\}}
            {\rm E}(1_{E_{n,i_1}}\cdots 1_{E_{n,i_k}}).
   \end{eqnarray*}
   For each fixed $n$, the events $E_{n,i}$, $i\in\{1,\ldots,n\}$, are
   exchangeable. Therefore,
   $$
   {\rm E}(M_n^k)
   \ =\ \sum_{j=1}^k S(k,j)(n)_j P(E_{n,1}\cap\cdots\cap E_{n,j}),
   $$
   where $S(k,j)$ denotes the Stirling number of the second kind, i.e.
   the number of ways to partition a set with $k$ elements in
   $j$ non-empty subsets. Division by $n^k$ and taking the limit
   $n\to\infty$ yields for all $k\in\nz_0$
   \begin{eqnarray*}
      \lim_{n\to\infty}{\rm E}\Big(\Big(\frac{M_n}{n}\Big)^k\Big)
      & = & \sum_{j=1}^k S(k,j)\lim_{n\to\infty}\frac{(n)_j}{n^k}
            P(E_{n,1}\cap\cdots\cap E_{n,j})\\
      & = & \lim_{n\to\infty} P(E_{n,1}\cap\cdots\cap E_{n,k})\\
      & = & E((1-e^{-rL_1})\cdots(1-e^{-rL_k}))\ =:\ \mu_k.
   \end{eqnarray*}
   For all $m,k\in\nz_0$,
   \begin{eqnarray*}
      \sum_{j=0}^m {m\choose j}(-1)^j\mu_{k+j}
      \ =\ \lim_{n\to\infty}
           {\rm E}\left(
              \sum_{j=0}^m {m\choose j}(-1)^j\Big(\frac{M_n}{n}\Big)^{k+j}
           \right)\\
      \ =\ \lim_{n\to\infty}
           {\rm E}\left(
              \Big(\frac{M_n}{n}\Big)^k\Big(1-\frac{M_n}{n}\Big)^m
           \right)
      \ \geq\ 0.
   \end{eqnarray*}
   Thus (Hausdorff moment problem), the sequence
   $(\mu_k)_{k\in\nz_0}$ is a moment sequence of some random variable
   $M$ taking values in the unit interval $[0,1]$. The convergence of
   moments implies the convergence $M_n/n\stackrel{d}{\to} M$.\hfill$\Box$
\end{proof}
\begin{remark}
   There is the following interpretation of the distribution of the
   limiting external branch lengths
   $L_i$, $i\in\nz$, in terms of the frequency spectrum of the coalescent.
   Let $S_t$ denote the frequency of
   singletons of $R_t$. Conditional on $S_{t_1},\ldots,S_{t_k}$, the
   probability that $i$ is still a singleton at time $t_i$,
   $i\in\{1,\ldots,k\}$, is
   $S_{t_1}\cdots S_{t_k}$. Therefore,
   for $t_1,\ldots,t_k\in [0,\infty)$,
   $$
   P(L_1>t_1,\ldots,L_k>t_k)\ =\ {\rm E}(S_{t_1}\cdots S_{t_k}),
   $$
   or, equivalently (in agreement with the principle of inclusion and exclusion),
   $$
   P(L_1\le t_1,\ldots,L_k\le t_k)
   \ =\ {\rm E}((1-S_{t_1})\cdots(1-S_{t_k})).
   $$
   Thus, the distribution function of $(L_1,\ldots,L_k)$ can be expressed
   in terms of the process $S=(S_t)_{t\ge 0}$.
\end{remark}
The following Corollary \ref{char} expresses the distribution
of the limiting random variable $M$ appearing in Lemma \ref{mnconv}
in terms of the process $(S_t)_{t\ge 0}$. There is the following rough
intuition for the form of the integral in Corollary \ref{char}. A
contribution to $M_n$ occurs every time a lineage that has not yet
coalesced experiences its first mutation.
The time of a first mutation is exponentially distributed with parameter $r$,
so at each time $t$ the infinitesimal growth of $M_n$ due to a not yet
coalesced lineage is $re^{-rt}$. Since $S_t$ is the fraction of singletons
at time $t$, the infinitesimal growth of $M_n$ at time $t$ is approximately
$re^{-rt}nS_t$. 
In \cite{moehlesites}, when the number of
segregating sites is the quantity of interest, any mutation contributes
to the count rather than just the first one, so we get $r$ in Proposition
5.1 of \cite{moehlesites} in place of the $re^{-rt}$ in Corollary \ref{char}.
\begin{corollary} \label{char}
   The limiting variable $M$ appearing in Lemma \ref{mnconv} satisfies
   $$
   M\ \stackrel{d}{=}\ r\int_0^\infty e^{-rt}S_t\,dt.
   $$
\end{corollary}
\begin{proof}
   Fix $k\in\nz$ and define $g:\rz^k\to\rz$ via $g(t):=(-1)^ke^{-r(t_1+\cdots+t_k)}$
   for $t=(t_1,\ldots,t_k)\in\rz^k$. Note that
   $h(t):=\frac{\partial^k}{\partial t_1\cdots\partial t_k}g(t)
   =r^ke^{-r(t_1+\cdots+t_k)}$.
   For $x=(x_1,\ldots,x_k),y=(y_1,\ldots,y_k)\in\rz^k$ with
   $x_i\le y_i$ for all $1\le i\le k$ define
   $$
   \Delta_x^yg\ :=\
   \sum_{\varepsilon_1,\ldots,\varepsilon_k\in\{0,1\}}
   (-1)^{\varepsilon_1+\cdots+\varepsilon_k}g(\varepsilon_1x_1+(1-\varepsilon_1)y_1,
   \ldots,\varepsilon_kx_k+(1-\varepsilon_k)y_k).
   $$
   With the notation $L:=(L_1,\ldots,L_k)$ we have
   \begin{eqnarray*}
      {\rm E}(M^k)
      & = & {\rm E}((1-e^{-rL_1})\cdots(1-e^{-rL_k}))
      \ = \ {\rm E}(\Delta_0^Lg)\\
      & = & \int_{\rz_+^k} \Delta_0^y g\,P_L(dy)
      \ = \ \int_{\rz_+^k} \int_{\rz_+^k} 1_{[0,y)}(t)h(t)
            \,\lambda^k(dt)\,P_L(dy).
   \end{eqnarray*}
   An application of Fubini's theorem yields
   \begin{eqnarray*}
      {\rm E}(M^k)
      & = & \int_{\rz_+^k} h(t)\int_{\rz_+^k} 1_{(t,\infty)}(y) P_L(dy)\,\lambda^k(dt)
      \ = \ \int_{\rz_+^k} h(t) P(L>t)\,\lambda^k(dt)\\
      & = & \int_{\rz_+^k} r^ke^{-r(t_1+\cdots+t_k)}
            {\rm E}(S_{t_1}\cdots S_{t_k})\,\lambda^k(dt_1,\ldots,dt_k)\\
      & = & {\rm E}\left(\left(
               \int_0^\infty re^{-rt}S_t\,dt
            \right)^k\right).
   \end{eqnarray*}
   Thus, the moments of the random variables $M$ and
   $\int_0^\infty re^{-rt}S_t\,dt$ coincide.
   As both random variables take almost surely
   values in the unit interval $[0,1]$, they are equal in
   distribution.\hfill$\Box$
\end{proof}
The moments of $M$ can be expressed in terms of the measure $\Xi$ as follows.
\begin{remark}
   Assume that the measure $\Xi$ of the exchangeable coalescent
   $(R_t)_{t\ge 0}$ satisfies (\ref{cond}). From the Poisson
   construction of the $\Xi$-coalescent (see Schweinsberg \cite{schweinsberg2})
   it follows that the process $X=(X_t)_{t\ge 0}$, defined via
   $X_t:=-\log S_t$ for $t\ge 0$, is a drift-free subordinator with Laplace
   exponent
   $$
   \Phi(\eta)\ =\ \int_{\Delta\setminus\{0\}}\frac{1-(1-|x|)^\eta}{(x,x)}\,\Xi(dx),
   \qquad\eta\ge 0.
   $$
   Note that, for $\eta\in\nz$, $e^{-t\Phi(\eta)}
   ={\rm E}(e^{-\eta X_t})={\rm E}(S_t^\eta)$
   is the probability that $\{1\},\ldots,\{\eta\}$ are (singleton) blocks
   of $R_t$. The L\'evy measure $\varrho$ on $(0,\infty]$ of the subordinator
   $X$ is hence the image of the measure $\nu(dx):=\Xi(dx)/(x,x)$
   via the transformation $T(x):=-\log(1-|x|)$, i.e. $\varrho(A)=
   \int_{T^{-1}(A)}(x,x)^{-1}\Xi(dx)$ for all Borel subsets $A$ of
   $(0,\infty]$. This result is in agreement with Proposition~26
   of Pitman \cite{pitman} for the special situation when the coalescent
   allows only for multiple collisions ($\Lambda$-coalescent). From
   $$
   \int_{\Delta\setminus\{0\}}\frac{|x|}{(x,x)}\Xi(dx)
   \ =\ \int_{\Delta\setminus\{0\}}|x|\nu(dx)
   \ =\ \int_{(0,\infty]}(1-e^{-y})\,\varrho(dy)
   $$
   and $(1-e^{-1})\min(y,1)\le 1-e^{-y}\le \min(y,1)$, $y\ge 0$, it follows
   that (\ref{cond}) is equivalent to
   \begin{equation}
      \varrho(\{0\})\ =\ 0
      \quad\mbox{and}\quad
      \int_{(0,\infty]} \min(y,1)\,\varrho(dy)\ <\ \infty.
   \end{equation}
   Note that the finiteness of the last integral is
   the typical condition for a measure $\varrho$ for being a
   L\'evy measure of some subordinator.
   From Proposition 3.1 of \cite{carmonapetityor} it follows that
   $M\stackrel{d}{=}r\int_0^\infty e^{-rt-X_t}dt$ has moments
   \begin{equation} \label{moments}
      {\rm E}(M^k)\ =\ \frac{r^kk!}{(r+\Phi(1))(2r+\Phi(2))\cdots(kr+\Phi(k))},
      \quad k\in\nz.
   \end{equation}
   In particular,
   \begin{eqnarray}
      {\rm Var}(M)
      & = & {\rm E}(M^2)-({\rm E}(M))^2\nonumber\\
      & = & \frac{2r^2}{(r+\Phi(1))(2r+\Phi(2))}- \frac{r^2}{(r+\Phi(1))^2}\nonumber\\
      & = & \frac{2r^2(r+\Phi(1))-r^2(2r+\Phi(2))}{(r+\Phi(1))^2(2r+\Phi(2))}\nonumber\\
      & = & \frac{r^2}{(r+\Phi(1))^2(2r+\Phi(2))}
            \int_{\Delta\setminus\{0\}}\frac{|x|^2}{(x,x)}\,\Xi(dx),
            \label{variance}
   \end{eqnarray}
   as $2\Phi(1)-\Phi(2)=\int_{\Delta\setminus\{0\}}|x|^2/(x,x)\Xi(dx)$.
\end{remark}
In the final remark of this section a distributional fixed-point
equation for $M$ is derived for $\Xi$-coalescents satisfying
\begin{equation} \label{cond3}
   \Xi(\{0\})\ =\ 0\quad\mbox{and}\quad
   \int_{\Delta\setminus\{0\}}\frac{\Xi(dx)}{(x,x)}\ <\ \infty.
\end{equation}
In the spirit of Bertoin and Le~Gall \cite{bertoinlegall} we call
measures $\Xi$ satisfying (\ref{cond3}) simple measures.
Note that (\ref{cond3}) implies (\ref{cond}).
\begin{remark}
   If (\ref{cond3}) holds, then the L\'evy measure $\varrho$ of the
   subordinator $X=(X_t)_{t\ge 0}$ is finite ($m_0:=\varrho((0,\infty])
   =\nu(\Delta\setminus\{0\})<\infty$),
   which means that $X$ is a compound Poisson process
   $X_t=\sum_{i=1}^{N_t}\eta_i$, where $N:=(N_t)_{t\ge 0}$
   is a homogeneous Poisson process with parameter $m_0$ and
   $\eta_i$, $i\in\nz$, are random variables, independent of each
   other and of $N$, with common distribution function
   $y\mapsto P(\eta_i\leq y)=m_0^{-1}\varrho((0,y])$. Let
   $T_1<T_2<T_3<\cdots$ denote the jump times of the Poisson process $N$.
   Note that $T_{i+1}-T_i$ is exponentially distributed with parameter $m_0$.
   We have
   \begin{eqnarray*}
      M
      & \stackrel{d}{=} & \int_0^\infty re^{-rt}S_t\,dt
      \ =\ \sum_{i=0}^\infty \int_{T_i}^{T_{i+1}} re^{-rt} S_t\,dt\\
      & = & \int_0^{T_1} re^{-rt}\,dt + e^{-\eta_1}\int_{T_1}^{T_2} re^{-rt}\,dt
            + e^{-\eta_1-\eta_2}\int_{T_2}^{T_3} re^{-rt}\,dt + \cdots \\
      & = & (1-e^{-rT_1}) + e^{-\eta_1}(e^{-rT_1}-e^{-rT_2}) + e^{-\eta_1-\eta_2}
            (e^{-rT_2}-e^{-rT_3}) + \cdots\\
      & = & (1-e^{-rT_1}) + e^{-\eta_1}e^{-rT_1}\cdot\\
      &   & \hspace{1cm}\cdot
            \Big(
               (1-e^{-r(T_2-T_1)}) + e^{-\eta_2}(e^{-r(T_2-T_1)}-e^{-r(T_3-T_1)})
               + \cdots
            \Big)\\
      & = & B + A(1-B)M_1,
   \end{eqnarray*}
   with $A:=e^{-\eta_1}$, $B:=1-e^{-rT_1}$ and $M_1\stackrel{d}{=}M$. Thus,
   $M$ satisfies the distributional fixed-point equation
   \begin{equation} \label{fixedpoint}
      M\ \stackrel{d}{=}\ B + A(1-B)M,
   \end{equation}
   where $A$ and $B$ are independent (and independent of $M$),
   $B$ is beta distributed
   with parameters $1$ and $m_0/r$, i.e., $P(B>x)=(1-x)^{m_0/r}$, $x\in (0,1)$,
   and the distribution of $1-A$ is the image of the measure $\nu_0:=\nu/m_0$
   under the transformation $|.|:\Delta\setminus\{0\}\to (0,1]$, $x\mapsto |x|$.
   Using an argument similar to that of Vervaat \cite{vervaat} shows
   that the distribution of $M$ is uniquely determined by the fixed-point
   equation (\ref{fixedpoint}). The distribution of $M$ coincides with the
   stationary distribution of the process $(Y_n)_{n\in\nz_0}$ recursively
   defined by $Y_0:=0$ and $Y_{n+1}:=A_n(1-B_n)Y_n+B_n$, where
   $((A_n,B_n))_{n\in\nz_0}$ is a sequence of independent, identically
   distributed random variables with $(A_n,B_n)\stackrel{d}{=}(A,B)$.
   Note that
   $$
   Y_n\ =\ \sum_{i=0}^{n-1} B_{n-i-1}\prod_{j=n-i}^{n-1}A_j(1-B_j)
   \ \stackrel{d}{=}\ \sum_{i=0}^{n-1}B_i\prod_{j=0}^{i-1}A_j(1-B_j),
   \quad n\in\nz_0,
   $$
   and, hence, that $M\stackrel{d}{=}\sum_{i=0}^\infty B_i\prod_{j=0}^{i-1}A_j(1-B_j)$.
\end{remark}
\subsection{The total number of mutated branches} \label{nn}
\setcounter{theorem}{0}
%
In order to analyze the total number $N_n$ of mutated branches we need to
study $C_n$, the number of collision events that take place in the
restricted coalescent process $(\varrho_nR_t)_{t\ge 0}$ until there
is just a single block. Note that, in general, $C_n\ge X_n$, the number
of jumps. For $\Lambda$-coalescents we have $C_n=X_n$.
\begin{lemma} \label{lemma2}
   Let $R$ be a $\Xi$-coalescent. If (\ref{cond}) holds, then
   $C_n/n\to 0$ in $L^1$.
\end{lemma}
\begin{proof}
   For $n\in\nz$ define $a_n:={\rm E}(C_n)$ for convenience. Note that
   the sequence $(a_n)_{n\in\nz}$ satisfies the recursion $a_1=0$ and
   $a_n=v_n+\sum_{k=1}^{n-1}r_{nk}a_k$ for $n\in\{2,3,\ldots\}$
   with $r_{nk}:=P(I_n=k)$, $n,k\in\nz$, $k<n$ and $v_n:={\rm E}(V_n)$,
   where $V_n$ denotes the number of internal branches starting at
   the time of the first jump of the restricted coalescent $(\varrho_nR_t)_{t\ge 0}$.

   We verify the convergence $C_n/n\to 0$ in $L^1$ by contradiction
   in analogy to Gnedin's proof of Proposition 3 in \cite{gnedin}.
   Note that a similar argument is used on p.~219 of
   \cite{iksanovmoehle2}. Assume that there exists $\varepsilon>0$
   such that $a_n>n\varepsilon$ for infinitely many values of $n$. Selecting
   $\varepsilon$ smaller, for any fixed $c$ we can obtain the inequality
   $a_n>\varepsilon n+c$ for infinitely many values of $n$. Let $n_c$ be
   the minimum such $n$. Then $n_c\to\infty$ as $c\to\infty$. For
   $k<n_c$ we have $a_k\leq\varepsilon k+c$ which implies
   \begin{eqnarray*}
      \varepsilon n_c+c
      & < & a_{n_c}
      \ = \ v_{n_c}+\sum_{k=1}^{n_c-1} r_{n_c,k}a_k\\
      & \leq & v_{n_c} + c + \varepsilon
   \sum_{k=1}^{n_c-1} kr_{n_c,k}\ =\ v_{n_c}+c+\varepsilon{\rm E}(I_{n_c}).
   \end{eqnarray*}
   The constant $c$ cancels and it follows that
   $\varepsilon{\rm E}(n_c-I_{n_c})<v_{n_c}$. For $c\to\infty$ we
   obtain the promised contradiction, as ${\rm E}(n-I_n)/v_n\to\infty$
   as $n\to\infty$ by Corollary \ref{appendix5} given in the appendix.
   Thus, for all $\varepsilon>0$ there exists $n_0=n_0(\varepsilon)\in\nz$
   such that $a_n/n\leq\varepsilon$ for all $n\ge n_0$.
   In other words, $a_n/n\to 0$ as $n\to\infty$.\hfill$\Box$
\end{proof}
We are now able to show that, if (\ref{cond}) holds, then the total
number $N_n$ of mutated branches and the number $K_n$ of types both
have the same asymptotic behavior as $M_n$ as $n\to\infty$.
\begin{corollary} \label{lastcor}
   Let $(R_t)_{t\ge 0}$ be a $\Xi$-coalescent with mutation rate
   $r>0$ satisfying (\ref{cond}). Then, $N_n/n\stackrel{d}{\to}M$ and
   as well $K_n/n\stackrel{d}{\to}M$, where $M$ is the random variable
   defined in Corollary \ref{char} with moments (\ref{moments}).
\end{corollary}
\begin{proof}
   We have $M_n\le K_n\le N_n+1$. Thus, by Lemma \ref{mnconv}, it suffices
   to verify that $(N_n+1-M_n)/n\to 0$ in probability. We even show
   that $(N_n+1-M_n)/n\to 0$ in $L^1$. We have
   \begin{eqnarray*}
      0
      & \leq & K_n-M_n
      \ \leq \ N_n+1-M_n\\
      & = & \mbox{number of non-external mutated branches}+1\\
      & \leq & \mbox{number of non-external branches}+1
      \ = \ C_n.
   \end{eqnarray*}
   It remains to note that $C_n/n\to 0$ in $L^1$ by Lemma
   \ref{lemma2}.\hfill$\Box$
\end{proof}
Note that Corollary \ref{lastcor} in particular finishes the proof
of Theorem \ref{main2}.
\subsection{Examples} \label{examples}
In this section we apply Theorem \ref{main2} to some concrete examples.

\vspace{2mm}

{\bf Example 1.} (Dirac coalescents)
   Fix a point $c\in\Delta\setminus\{0\}$ and suppose that $\Xi=\delta_c$
   is the
   Dirac measure in $c$. Then, condition (\ref{cond}) holds, as
   $\int_{\Delta\setminus\{0\}} (|x|/(x,x))\Xi(dx)=|c|/(c,c)<\infty$. By Theorem
   \ref{main2}, all three random variables, $M_n/n$, $K_n/n$ and
   $N_n/n$ converge in distribution to $M:=r\int_0^\infty e^{-rt-X_t}dt$,
   where $X=(X_t)_{t\ge 0}$ is a subordinator with Laplace exponent
   $\Phi(\eta)=(1-(1-|c|)^\eta)/(c,c)$, $\eta\ge 0$. The L\'evy measure
   $\varrho=(1/(c,c))\delta_{-\log(1-|c|)}$ is hence the Dirac measure
   in $-\log(1-|c|)$ scaled by the factor $1/(c,c)$. We have $\Phi(1)=|c|/(c,c)$
   and $\Phi(2)=|c|(2-|c|)/(c,c)$ and, therefore, by (\ref{moments})
   and (\ref{variance}),
   \begin{equation} \label{deltamean}
      {\rm E}(M)\ =\ \frac{r}{r+\frac{|c|}{(c,c)}}
   \end{equation}
   and
   \begin{equation} \label{deltavariance}
      {\rm Var}(M)
      \ =\ \frac{r^2|c|^2/(c,c)}{(r+\frac{|c|}{(c,c)})^2(2r+\frac{|c|(2-|c|)}{(c,c)})}.
   \end{equation}
   Note that $m_0:=\int_{\Delta\setminus\{0\}}(1/(x,x))\Xi(dx)=1/(c,c)<\infty$, i.e.
   (\ref{cond3}) holds as well. Thus, by (\ref{fixedpoint}), $M$
   satisfies the distributional fixed-point equation
   $M\stackrel{d}{=}B+(1-|c|)(1-B)M$, where $B$ is a random variable
   independent of $M$ and beta-distributed with parameters $1$ and
   $m_0/r=1/((c,c)r)$. Even for this quite simple situation of Dirac coalescents,
   it does not seem to be straightforward to find simpler
   characterizations for the distribution of $M$.
%

\vspace{2mm}

{\bf Example 2.} (beta-coalescents)
   Let $\Lambda$ be beta distributed with parameters $a>1$ and $b>0$, i.e.,
   $\Lambda$ has density $u\mapsto (B(a,b))^{-1}u^{a-1}(1-u)^{b-1}$,
   $u\in (0,1)$, with
   respect to the Lebesgue measure on $(0,1)$, where $B(.,.)$ denotes the
   beta function. In this situation we have
   $$
   \int_{[0,1]}u^{-1}\,\Lambda(du)
   \ =\ \frac{B(a-1,b)}{B(a,b)}
   \ =\ \frac{a+b-1}{a-1}\ <\ \infty.
   $$
   Thus, Theorem \ref{main2} is applicable and
   all three random variables $M_n/n$, $K_n/n$
   and $N_n/n$, converge in distribution to $M:=r\int_0^\infty e^{-rt-X_t}dt$
   as $n\to\infty$, where $X=(X_t)_{t\ge 0}$ is a subordinator with Laplace
   exponent
   $$
   \Phi(\eta)
   \ =\ \frac{1}{B(a,b)}\int_0^1\frac{1-(1-u)^\eta}{u^2}u^{a-1}(1-u)^{b-1}\,du,
   \qquad\eta\ge 0.
   $$
   The expansion $1-(1-u)^\eta=\sum_{i=1}^\infty {\eta\choose i}(-1)^{i+1}u^i$
   yields
   \begin{eqnarray*}
      \Phi(\eta)
      & = & \frac{1}{B(a,b)}\sum_{i=1}^\infty {\eta\choose i}(-1)^{i+1}
            B(a+i-2,b)\\
      & = & \frac{a+b-1}{a-1}\sum_{i=1}^\infty
               {\eta\choose i}(-1)^{i+1}\prod_{j=1}^{i-1}\frac{a-2+j}{a+b-2+j},
           \qquad\eta\ge 0.
   \end{eqnarray*}
   Note that $\Phi(1)=(a+b-1)/(a-1)$ and $\Phi(2)=(a+2b-1)/(a-1)$.
   The mean and the variance of $M$ can be easily deduced from
   (\ref{moments}) and (\ref{variance}).


   From $\varrho((0,y])=\nu((0,1-e^{-y}])=\int_{(0,1-e^{-y}]} u^{-2}\Lambda(du)$
   it follows that the L\'evy measure $\varrho$ of the subordinator $X$
   has density $y\mapsto (B(a,b))^{-1}(1-e^{-y})^{a-3}(e^{-y})^b$, $y\in (0,\infty)$,
   with respect to the Lebesgue measure on $(0,\infty)$. If $a>2$, then
   $$
   m_0\ :=\ \int_{[0,1]}u^{-2}\Lambda(du)
   \ =\ \frac{(a+b-1)(a+b-2)}{(a-1)(a-2)}\ <\ \infty.
   $$
   In this case, by (\ref{fixedpoint}),
   $M$ satisfies the distributional fixed-point equation
   $M\stackrel{d}{=}B+A(1-B)M$, where $A$ and $B$ are independent (and
   independent of $M$), $1-A$ is beta distributed with parameters
   $a-2$ and $b$, and $B$ is beta distributed with parameters
   $1$ and $m_0/r$.

   For special parameter values of $a$ and $b$ the Laplace exponent $\Phi$ can be
   further simplified. For example, for the $\beta(2-\alpha,\alpha)$-coalescent
   with $0<\alpha<1$,
   $$
   \Phi(\eta)
   \ =\ \frac{1}{1-\alpha}\sum_{i=1}^\infty{\eta\choose i}{{\alpha-1}\choose{i-1}}
   \ =\ \frac{\eta\Gamma(\eta+\alpha)}{(1-\alpha)\Gamma(\alpha+1)\Gamma(\eta+1)},
   \qquad\eta\ge 0.
   $$
   Note that, if the conjecture on p.~495 of Basdevant and
   Goldschmidt \cite{basdevantgoldschmidt} is correct, then we have identified
   (in the notation of \cite{basdevantgoldschmidt}) the distribution of the
   random variable $C_1$, namely $C_1\stackrel{d}{=}M$.

\vspace{2mm}

{\bf Example 3.}
   Suppose that the measure $\Xi$ is concentrated on the subset $\Delta^*$
   of all points $x\in\Delta$ satisfying $|x|=1$ and that
   $m_0:=\int_{\Delta\setminus\{0\}}(1/(x,x))\,\Xi(dx)<\infty$.
   Concrete examples are
   the star-shaped coalescent, where $\Xi$ is the Dirac measure in $(1,0,0,\ldots)$,
   or the Poisson-Dirichlet coalescent with parameter $\theta>0$, where
   $\Xi$ is assumed to have density $x\mapsto
   (x,x)$ with respect to the Poisson-Dirichlet distribution with parameter
   $\theta>0$. Then, (\ref{cond}) and (\ref{cond3}) coincide and are both
   satisfied. Thus, Theorem \ref{main2} is applicable, i.e., all three
   random variables, $M_n/n$, $K_n/n$ and $N_n/n$, converge in distribution
   to a limiting variable $K$ with moments (\ref{kj}). As the measure $\Xi$
   is concentrated on $\Delta^*$, the Laplace exponent $\Phi(\eta)\equiv m_0$
   is constant. Therefore, $K$ has moments
   ${\rm E}(K^j)=r^jj!/((r+m_0)\cdots(jr+m_0))$, $j\in\nz$.
   It follows that $K$ is beta-distributed with parameters $1$ and $m_0/r$.

\subsection{Appendix}
\setcounter{theorem}{0}
In this appendix basic results for $\Xi$-coalescents
$R=(R_t)_{t\ge 0}$ are derived. We first restrict our attention to
coalescents with (only) multiple collisions, as the proofs are in this
case less technical. Afterwards we extend the results to
$\Xi$-coalescents. Our first result (Lemma \ref{appendix1}) concerns
the number of blocks $I_n$ of the restricted coalescent process
$(\varrho_nR_t)_{t\ge 0}$ after its first jump. Note that $I_n$ has
distribution (\ref{indist}) and that we define $I_1:=0$ for convenience.
Lemma \ref{appendix1} is well known from the literature (see, for example,
Schweinsberg \cite[Lemma 3]{schweinsberg1}), however, we provide a proof
which can be extended to the full class of coalescents with
simultaneous multiple collisions (see Lemma \ref{appendix3}).
\begin{lemma} \label{appendix1}
   Let $R$ be a $\Lambda$-coalescent. Then, for all $n\in\nz$,
   $$
   g_n{\rm E}(n-I_n)\ =\ \int_{[0,1]}\frac{(1-u)^n-1+nu}{u^2}\Lambda(du)
   $$
   with continuous extension of the function below the integral
   for $u\searrow 0$.
\end{lemma}
\begin{proof}
   We have
   $$
   g_n{\rm E}(I_n)
   \ =\ \sum_{k=1}^{n-1} kg_{nk}
   \ =\ \sum_{k=1}^{n-1} k\int_{[0,1]}{n\choose{k-1}}u^{n-k-1}(1-u)^{k-1}\Lambda(du).
   $$
   Substituting $i=k-1$ and interchanging the summation with the integral yields
   \begin{eqnarray*}
   g_n{\rm E}(I_n)
   & = & \int_{[0,1]}
         \sum_{i=0}^{n-2} (i+1){n\choose i}u^{n-i}(1-u)^i
         \frac{\Lambda(du)}{u^2}\\
   & = & \int_{[0,1]}
         \frac{n(1-u)+1-n^2u(1-u)^{n-1}-(n+1)(1-u)^n}{u^2}
         \,\Lambda(du)\\
   & = & \int_{[0,1]}
            \frac{n+1-nu - n^2u(1-u)^{n-1}-(n+1)(1-u)^n}{u^2}
         \,\Lambda(du).\\
   \end{eqnarray*}
   Now subtract this expression from
   $$
   ng_n\ =\ \int_{[0,1]}\frac{n-n(1-u)^n-n^2u(1-u)^{n-1}}{u^2}\,\Lambda(du).
   \hspace{3cm}\Box
   $$
\end{proof}
\begin{corollary} \label{appendix2}
   If (\ref{cond2}) holds, then ${\rm E}(n-I_n)\sim n/g_n
   \int_{[0,1]}u^{-1}\Lambda(du)\to\infty$ as $n\to\infty$.
\end{corollary}
\begin{proof}
   For $n\in\nz$ define the auxiliary function $H(n):=\int_{[0,1]}(1-(1-u)^n)u^{-2}\Lambda(du)$.
   Note that $1-(1-u)^n\le nu$ for $n\in\nz$ and $u\in [0,1]$ and therefore
   $H(n)\le n\int_{[0,1]}\Lambda(du)/u=nH(1)<\infty$ for all
   $n\in\nz$. By Lemma \ref{appendix1}, $g_n{\rm E}(n-I_n)=nH(1)-H(n)$.
   If we can show that $g_n/n\to 0$ and that $H(n)/n\to 0$ as $n\to\infty$,
   then,
   $$
   {\rm E}(n-I_n)\ =\ \frac{n}{g_n}\Big(H(1)-\frac{H(n)}{n}\Big)
   \ \sim\ \frac{n}{g_n}H(1)\ \to\ \infty
   $$
   and we are done. Since
   \begin{eqnarray*}
      g_n
      & = & \int_{[0,1]}\frac{1-(1-u)^n-nu(1-u)^{n-1}}{u^2}\Lambda(du)\\
      & \le & \int_{[0,1]}\frac{1-(1-u)^n}{u^2}\Lambda(du)\ =\ H(n),
   \end{eqnarray*}
   it remains to verify that $H(n)/n\to 0$ as $n\to\infty$.
   By assumption, the measure $\mu(du):=\Lambda(du)/u$ is finite and has
   no mass at zero. We have
   $$
   \frac{H(n)}{n}
   \ =\ \int_{[0,1]}\frac{1-(1-u)^n}{nu}\frac{\Lambda(du)}{u}
   \ =\ \int_{[0,1]}f_n(u)\mu(du),
   $$
   where $f_n(u):=(1-(1-u)^n)/(nu)$ for $n\in\nz$ and $u\in[0,1]$.
   Obviously, $0\le f_n\le 1$ for all $n\in\nz$ and $f_n$ converges
   pointwise to zero on $(0,1]$ as $n\to\infty$.
   Thus, $H(n)/n\to 0$ as $n\to\infty$ by dominated convergence.\hfill$\Box$
\end{proof}
In the following Lemma \ref{appendix1} is extended to $\Xi$-coalescents.
\begin{lemma} \label{appendix3}
   Let $\Xi=a\delta_0+\Xi_0$ be a finite measure on the infinite simplex
   $\Delta$ and let $(R_t)_{t\ge 0}$ be a $\Xi$-coalescent.
   For $n\in\nz$ let $I_n$ be the number of equivalence classes (blocks)
   of the restricted coalescent process $(\varrho_nR_t)_{t\ge 0}$ after
   its first jump ($I_1:=0$). Then, for all $n\in\nz$,
   \begin{equation} \label{appendix3formula}
      g_n{\rm E}(n-I_n)\ =\ a{n\choose 2} +
      \int_\Delta\bigg(
         n|x|-\sum_{i=1}^\infty\big(1-(1-x_i)^n\big)
      \bigg)
      \frac{\Xi_0(dx)}{(x,x)}.
   \end{equation}
\end{lemma}
\begin{proof}
   Fix $n\in\nz$. The first summand on the right hand side in
   (\ref{appendix3formula}) is obvious, because
   with probability $a=\Xi(\{0\})$, the coalescent behaves
   as the Kingman coalescent, in which case we have
   $I_n=n-1$ and $g_n={n\choose 2}$. Thus, without loss of
   generality we can and do assume that $a=0$.

   In the following we exploit Schweinsberg's \cite{schweinsberg2} Poisson
   process construction of exchangeable coalescents. Note that this
   construction is essentially equivalent to Kingman's \cite{kingman1} paintbox
   construction and closely related to the Bernoulli sieve \cite{gnedin}.

   For given $x\in\Delta$ partition $[0,1)$ into
   intervals $J_0,J_1,J_2,\ldots$ of lengths $x_0:=1-|x|,x_1,x_2,\ldots$,
   i.e., $J_0:=[0,x_0)$, $J_1:=[x_0,x_0+x_1)$, $J_2:=[x_0+x_1,x_0+x_1+x_2)$
   and so on. Let $U_1,\ldots,U_n$ be independent random variables uniformly
   distributed on $[0,1)$. For $i\in\nz_0$ let
   $$
   X_i\ :=\ X_i(n)\ :=\ \sum_{j=1}^n 1_{J_i}(U_j)
   $$
   denote the number of $U_1,\ldots,U_n$ which fall into the interval $J_i$.
   Note that $X_i$ is binomially distributed with parameters $n$ and $x_i$
   and that $\sum_{i=0}^\infty X_i=n$. Therefore,
   \begin{eqnarray*}
      &   & \hspace{-15mm}P\Big(\bigcap_{i\in\nz}\{X_i\le 1\}\Big)\\
      & = & P(X_0=n) +
            \sum_{l=1}^n \sum_{{i_1,\ldots,i_l\in\nz}\atop{\rm all\,distinct}}
            P(X_0=n-l,X_{i_1}=1,\ldots,X_{i_l}=1)\\
      & = & x_0^n + \sum_{l=1}^n {n\choose l}x_0^{n-l}
            \sum_{{i_1,\ldots,i_l\in\nz}\atop{\rm all\,distinct}}
            x_{i_1}\cdots x_{i_l}.
   \end{eqnarray*}
   We have
   \begin{eqnarray}
      &   & \hspace{-15mm}g_n{\rm E}(I_n)\nonumber\\
      & = & \sum_{k=1}^{n-1}kg_{nk}
      \ = \ \sum_{k=1}^{n-1}k\int_\Delta
            P\Big(X_0+\sum_{i=1}^\infty 1_{\{X_i\ge 1\}}=k\Big)\frac{\Xi_0(dx)}{(x,x)}
            \nonumber\\
      & = & \int_\Delta
            \sum_{k=1}^{n-1}kP\Big(X_0+\sum_{i=1}^\infty 1_{\{X_i\ge 1\}}=k\Big)
            \frac{\Xi_0(dx)}{(x,x)}\nonumber\\
      & = & \int_\Delta
            \bigg(
               {\rm E}(X_0)
               + \sum_{i=1}^\infty P(X_i\ge 1)
               - nP\Big(\bigcap_{i=1}^\infty \{X_i\le 1\}\Big)
            \bigg)\frac{\Xi_0(dx)}{(x,x)}. \label{gnein}
   \end{eqnarray}
   Now subtract this expression from (see Schweinsberg
   \cite[p.~36, Eq. (70)]{schweinsberg2})
   \begin{eqnarray*}
      ng_n
      & = & n\int_\Delta
            \Bigg(
               1-x_0^n-\sum_{l=1}^n {n\choose l}x_0^{n-l}
               \sum_{{i_1,\ldots,i_l\in\nz}\atop{\rm all\,distinct}}
               x_{i_1}\cdots x_{i_l}
            \Bigg)\frac{\Xi_0(dx)}{(x,x)}\\
      & = & \int_\Delta \Bigg(
               n-nP\Big(\bigcap_{i=1}^\infty\{X_i\le 1\}\Big)
            \Bigg)\frac{\Xi_0(dx)}{(x,x)}
   \end{eqnarray*}
   and note that ${\rm E}(X_0)=n(1-|x|)$ and that
   $P(X_i\ge 1)=1-(1-x_i)^n$.\hfill$\Box$
\end{proof}
For $n\in\nz\setminus\{1\}$ we now study the number
$V_n$ of internal branches of the
restricted coalescent process which start after the time $T_n$
of the first jump of the restricted coalescent process
$(\varrho_nR_t)_{t\ge 0}$. Note that $V_n=I_n-S_n$, where
$S_n$ denotes the number of singleton blocks of the restricted
coalescent process $(\varrho_nR_t)_{t\ge 0}$ after its first
jump.
\begin{lemma} \label{appendix4}
   For all $n\in\nz\setminus\{1\}$,
   \begin{equation} \label{appendix4formula}
      g_n{\rm E}(V_n)
      \ =\ a{n\choose 2} +
      \int_\Delta \sum_{i=1}^\infty
      \big(1-(1-x_i)^n-nx_i(1-x_i)^{n-1}\big)
      \frac{\Xi_0(dx)}{(x,x)}.
   \end{equation}
\end{lemma}
\begin{proof}
   Fix $n\in\nz\setminus\{1\}$.
   Again, without loss of generality we can and do assume that $a=0$.
   Using the notation of the previous proof it follows that
   \begin{eqnarray*}
      g_n{\rm E}(S_n)
      & = & \sum_{s=0}^{n-1} s\int_\Delta
               P\Big(X_0+\sum_{i=1}^\infty 1_{\{X_i=1\}}=s\Big)
            \frac{\Xi_0(dx)}{(x,x)}\\
      & = & \int_\Delta \sum_{s=0}^{n-1} s
            P\Big(X_0+\sum_{i=1}^\infty 1_{\{X_i=1\}}=s\Big)
            \frac{\Xi_0(dx)}{(x,x)}\\
      & = & \int_\Delta
            \bigg(
               {\rm E}(X_0)+\sum_{i=1}^\infty P(X_i=1)
               - n P\Big(\bigcap_{i=1}^\infty \{X_i\le 1\}\Big)
            \bigg)\frac{\Xi_0(dx)}{(x,x)}.
   \end{eqnarray*}
   If we subtract this quantity from the expression (\ref{gnein}) already
   derived for $g_n{\rm E}(I_n)$ we arrive at
   $$
   g_n{\rm E}(V_n)
   \ =\ \int_\Delta\sum_{i=1}^\infty P(X_i\ge 2)\frac{\Xi_0(dx)}{(x,x)}
   $$
   and the lemma follows from $P(X_i\ge 2)=
   1-(1-x_i)^n-nx_i(1-x_i)^{n-1}$.\hfill$\Box$
\end{proof}
\begin{remark}
   Fix $n\in\nz\setminus\{1\}$.
   For $\Lambda$-coalescents, (\ref{appendix4formula}) reduces to
   $$
   g_n{\rm E}(V_n)\ =\ \int_{[0,1]}(1-(1-x)^n-nx(1-x)^{n-1})\frac{\Lambda(dx)}{x^2}
   \ =\ g_n.
   $$
   Thus, ${\rm E}(V_n)=1$, which is clear, as $V_n\equiv 1$ for coalescents
   with only multiple (no simultaneous multiple) collisions.
\end{remark}
\begin{corollary} \label{appendix5}
   If (\ref{cond}) holds, then
   $\lim_{n\to\infty}{\rm E}(n-I_n)/{\rm E}(V_n)=\infty$.
\end{corollary}
\begin{proof}
   Define the auxiliary function $H:\nz\to\rz$ via
   $$
   H(n)\ :=\ \int_{\Delta\setminus\{0\}} \sum_{i=1}^\infty
   \big(1-(1-x_i)^n\big)\frac{\Xi(dx)}{(x,x)},\quad n\in\nz.
   $$
   Note that $1-(1-x_i)^n\le nx_i$ for $n\in\nz$ and $x_i\in[0,1]$, and,
   therefore,
   $$
   0\ <\ H(n)
   \ \le n\int_{\Delta\setminus\{0\}}
   |x|\frac{\Xi(dx)}{(x,x)}\ =\ nH(1)\ <\ \infty.
   $$
   We rewrite (\ref{appendix3formula}) in terms of
   the auxiliary function $H$ as $g_n{\rm E}(n-I_n)=nH(1)-H(n)$.
   Moreover, from (\ref{appendix4formula}) it follows
   that $g_n{\rm E}(V_n)\le H(n)$. Thus,
   $$
   \frac{{\rm E}(n-I_n)}{{\rm E}(V_n)}
   \ \geq\ \frac{nH(1)-H(n)}{H(n)}\ =\ \frac{nH(1)}{H(n)}-1.
   $$
   It remains to verify that $\lim_{n\to\infty}H(n)/n=0$. By assumption,
   the measure $\mu(dx):=(|x|/(x,x))\Xi(dx)$ is finite and has no mass at zero.
   We have
   $$
   H(n)\ =\ \int_{\Delta\setminus\{0\}} f_n(x)\mu(dx),
   $$
   where $f_n(x):=\sum_{i=1}^\infty(1-(1-x_i)^n)/(n|x|)$ for $n\in\nz$
   and $x\in\Delta\setminus\{0\}$. From $1-(1-x_i)^n\le nx_i$ for $x_i\in [0,1]$
   it follows that $0\le f_n\le 1$ for all $n\in\nz$. It is shown below that
   $f_n$ converges pointwise to zero on $\Delta\setminus\{0\}$ as $n\to\infty$.
   Therefore, $H(n)/n\to 0$ as $n\to\infty$ by dominated convergence and the
   corollary is established. In order to verify the pointwise convergence
   of $f_n$ to zero fix $x\in\Delta\setminus\{0\}$ and let
   $\delta_\nz$ denote the counting measure on $\nz$. We have
   $$
   |x|f_n(x)\ =\ \sum_{i=1}^\infty \frac{1-(1-x_i)^n}{n}
   \ =\ \int g_nd\delta_\nz
   $$
   with $g_n:\nz\to\rz$ defined via $g_n(i):=(1-(1-x_i)^n)/n$. Obviously
   $g_n\to 0$ pointwise as $n\to\infty$, as $0\le g_n\leq 1/n$ for all $n\in\nz$.
   Moreover, $g_n(i)\le x_i=:g(i)$ for all $n\in\nz$. The function $g$ is
   integrable with respect to the counting measure $\varepsilon_\nz$
   ($\int gd\delta_\nz=\sum_{i=1}^\infty x_i\le 1$). Thus, $f_n(x)\to 0$
   as $n\to\infty$ by dominated convergence.\hfill$\Box$
\end{proof}
\begin{acknowledgement}
   The authors thank Alex Iksanov for fruitful comments concerning the remark
   before Corollary \ref{char} and the final remark of Section \ref{mn}.
\end{acknowledgement}
%
%

\end{document}